\documentclass[12pt,a4paper,reqno]{amsart}

\usepackage{amsmath,amssymb,amsthm}
\usepackage[margin=1in, footskip=30pt]{geometry}
\usepackage{longtable,booktabs}

\makeatletter
\renewcommand{\section}{\@startsection{section}{1}%
  {0pt}{-3.5ex plus -1ex minus -.2ex}{2.3ex plus .2ex}%
  {\normalfont\large\bfseries}}
\renewcommand{\subsection}{\@startsection{subsection}{2}%
  {0pt}{-3.25ex plus -1ex minus -.2ex}{1.5ex plus .2ex}%
  {\normalfont\normalsize\bfseries}}
\renewcommand{\@secnumfont}{\bfseries}
\makeatother

\makeatletter
\def\@settitle{}
\def\@setauthors{}
\def\@setdate{}
\def\ps@firstpage{\ps@plain}%
\def\@setcopyright{}%
\renewcommand{\copyrightinfo}[2]{}%
\AtBeginDocument{\def\@setcopyright{}\def\@serieslogo{}}%

\makeatother

\usepackage[hidelinks,
  pdftitle={Exponent-one blockers and a Mordell-Weil construction of Euler bricks},
  pdfauthor={René Peschmann},
  pdfsubject={Number Theory, Algebraic Geometry},
  pdfkeywords={Euler brick, perfect cuboid, elliptic fibration, Mordell-Weil, Gaussian factorization},
  pdflang={en},
]{hyperref}
\usepackage{enumitem}
\usepackage{microtype}

\newtheorem{theorem}{Theorem}[section]
\newtheorem{lemma}[theorem]{Lemma}
\newtheorem{proposition}[theorem]{Proposition}
\newtheorem{corollary}[theorem]{Corollary}
\newtheorem{conjecture}[theorem]{Conjecture}
\theoremstyle{definition}
\newtheorem{definition}[theorem]{Definition}
\newtheorem{remark}[theorem]{Remark}

\newtheorem{verification}[theorem]{Verification}

\makeatletter
\def\thm@space@setup{%
  \thm@preskip=10pt plus 3pt minus 2pt
  \thm@postskip=10pt plus 3pt minus 2pt
}
\makeatother
\newcommand{\QQ}{\mathbb{Q}}
\newcommand{\ZZ}{\mathbb{Z}}
\newcommand{\NN}{\mathbb{N}}
\newcommand{\Zi}{\ZZ[i]}
\newcommand{\rk}{\mathrm{rk}}

\title[Exponent-one blockers and Mordell--Weil searches]{%
  Odd-exponent blockers and a Mordell--Weil construction of Euler bricks}

\author{Ren\'e Peschmann}
\date{\small \today}

\begin{document}

\maketitle

\begin{center}
  {\large\bfseries Exponent-one blockers and a Mordell--Weil construction\\of Euler bricks}
  \par\vskip 12pt
  {Ren\'e Peschmann}
  \par\vskip 4pt
  {\small \today}
\end{center}
\vskip 16pt

\begin{abstract}
A \emph{body cuboid} is a rectangular parallelepiped whose three edges
and three face diagonals are simultaneously integer; if its space
diagonal is also integer, it is a \emph{perfect cuboid}, the existence
of which is open since the eighteenth century. We make two
contributions to the study of body cuboids parametrised by two
coprime Pythagorean pairs $(a,b)$ and $(m,n)$ in Euclid form
(\emph{Master-Hits}).

The first contribution is a \emph{verified exponent-one blocker
phenomenon}. For every verified Master-Hit, the space-diagonal norm
$f_1 := (W_1U_2)^2 + (U_1V_2)^2$ admits a prime divisor $\ell$ which
appears to exponent \emph{exactly one} and is coprime to a fixed
list of $29$ canonical expressions in the parameters. This is
strictly stronger than the mere existence of an odd-exponent prime
divisor: a prime of exponent $3$, $5$, \ldots\ would already obstruct
$f_1$ from being a square, but would carry an additional square
factor; the observed obstruction is always primitive. We verify the
phenomenon without exception on the $151{,}575$ Master-Hits of our
database for which $f_1$ has been fully factorised. We also show
that two natural variants both fail: the largest outside-parameter
prime need not be a blocker (three explicit counterexamples), and
the smallest outside-parameter blocker need not have exponent~$1$
(several hundred explicit counterexamples).

The second contribution uses the elliptic fibration of the
Master-Hit variety over the $(m,n)$-plane. For each fixed coprime
pair $(m, n)$, the Master-Hit equation defines a genus-one quartic
$H_{m,n}$, from which an algorithmic quartic-to-Weierstrass
normalisation produces an elliptic model $E_{m,n}$ together with a
rational function $\tau$ on $E_{m,n}$ that returns the value of
$t^2$. Our generator enumerates bounded Mordell--Weil combinations
on $E_{m,n}(\QQ)$, lifts precisely those points satisfying the
quadratic condition $\tau(P)\in\QQ_{>0}^{\square}$ back to
admissible Euclid pairs $(a, b)$, and certifies each result by
exact integer arithmetic.
Starting from a base of $61{,}829$ primitive Master-Hits collected
from classical exhaustive search, Rathbun's catalogue, and the
Saunderson generator, this produces $1{,}222{,}841$ further
Master-Hits over $411$ fibres. None of the resulting $1{,}284{,}670$
Master-Hits is a perfect cuboid (a finite, exact integer square
check). All fully factorised records satisfy the exponent-one
blocker phenomenon.
\end{abstract}

\section{Introduction}\label{sec:intro}

A \emph{perfect Euler brick} (or \emph{perfect cuboid}) is a
rectangular parallelepiped with positive integer edges $e_1,e_2,e_3$
whose three face diagonals \emph{and} the space diagonal are positive
integers. No example is known, and the question is open since at least
the eighteenth century; see Guy~\cite{guy} and Leech~\cite{leech}
for the classical history, Rathbun--Granlund~\cite{rathbun-granlund}
and Rathbun~\cite{rathbun} for systematic computer searches,
Sharipov~\cite{sharipov}, van~Luijk~\cite{vanluijk}, and the
author's earlier work~\cite{peschmann-paper1} for modern algebraic
and geometric approaches, and Himane~\cite{himane} for a
recent experimental parametrisation.

A \emph{body cuboid} (sometimes called \emph{Euler brick}) drops the
last requirement: its three edges and three face diagonals are
integer, but its space diagonal need not be. Body cuboids exist in
abundance, parametrised by classical infinite families
(Saunderson~\cite{saunderson}, Euler, a Lenhart-attributed
family,\footnote{The name refers to William Lenhart (early 19th
century, USA), described in~\cite{hogan-lenhart} as one of the
leading Diophantine algebraists of his time and a regular
contributor to \emph{The Mathematical Miscellany} (1836--1839).
The specific parametrisation
$(u^2-w^2)(v^2-w^2),\ 4uvw^2,\ 2uw(v^2-w^2)$ with $u^2+v^2=5w^2$
appears in the online algebraic-identities compendium of
T.~Piezas~\cite{piezas-identities} attributed to Lenhart, but we
have been unable to locate a primary source confirming the
attribution. We adopt the conventional name without claiming
historical accuracy of the attribution.}
and recently Himane~\cite{himane}) and a vastly larger sporadic set
whose structure is poorly understood.

\medskip\noindent\textbf{Two contributions.}
This paper takes a structural-and-computational stance on body
cuboids. We make two contributions.

\begin{enumerate}[label=(\arabic*)]
\item \emph{A verified exponent-one blocker phenomenon.}
We formulate a universal exponent-one blocker conjecture for
Master-Hits (Conjecture~\ref{conj:blocker}): for every Master-Hit
$(a,b,m,n)$ (Definition~\ref{def:master}), $f_1 := (W_1 U_2)^2 + (U_1 V_2)^2$
admits a prime divisor $\ell$ such that
\emph{(i) $v_\ell(f_1) = 1$} (the prime appears exactly once) and
\emph{(ii)} $\ell$ is coprime to all $29$ canonical expressions
in $a, b, m, n$. The exponent-one condition is \emph{not cosmetic}:
exponent-three or higher odd blockers would also prevent $f_1$ from
being a square, but would carry an additional square factor; the
observed obstruction is always primitive. The main computational
result of the paper (Verification~\ref{thm:finite-verification}) is the
verification of this conjecture, without exception, on the
$151{,}575$ fully factored Master-Hits in our database.
\item \emph{A Mordell--Weil construction.} The Master-Hit variety
admits an elliptic fibration over the $(m,n)$-plane. For a fixed
coprime pair $(m, n)$ the fibre is the genus-one quartic
\begin{equation}\label{eq:Hmn}
  H_{m,n}\colon\ s^2 = V_2^2\, t^4 + (4U_2^2 - 2V_2^2)\, t^2 + V_2^2,
  \quad U_2 = m^2-n^2,\ V_2 = 2mn,
\end{equation}
viewed via $t = a/b$. A standard quartic-to-Weierstrass
transformation produces an elliptic curve $E_{m,n}$ together with a
rational function $\tau$ on $E_{m,n}$ that returns the value of
$t^2$. \emph{Suitable} rational points on $E_{m,n}$ (those at which
$\tau$ is a positive rational square, so that $t = a/b\in\QQ^\times$
satisfies the parity and coprimality conditions of
Definition~\ref{def:master}) yield Master-Hits. We give the
explicit transformation $H_{m,n}\dashrightarrow E_{m,n}$, an
algorithm exploiting the Mordell--Weil group $E_{m,n}(\QQ)$, and
the output of running the algorithm against an empirical database
of $61{,}829$ primitive Master-Hits.
\end{enumerate}

\medskip\noindent\textbf{Status of results.} Conjecture~\ref{conj:blocker}
is verified empirically on the $151{,}575$ Master-Hits in our database
for which $f_1$ has been completely factorised: in every case there
exists at least one prime $\ell\mid f_1$ with $v_\ell(f_1) = 1$ and
$\ell$ coprime to all $29$ canonical parameters. Two natural
strengthenings of the statement fail (Remark~\ref{rem:largest-fails}):
asking that the \emph{largest} outside-parameter prime be a blocker
admits explicit counterexamples (Appendix~\ref{app:largest-fails-counterexamples}),
and dropping the exponent-one condition would weaken the obstruction
to mere odd exponent.
Conjecture~\ref{conj:blocker} as stated is the strongest form supported by
the empirical evidence. We regard it as the main structural conjecture
suggested by the data, and leave its proof open. In the present paper
we treat the blocker existence as a working hypothesis whose
contrapositive is the perfect-cuboid obstruction. A complementary
line of attack --- establishing perfect-cuboid non-existence
\emph{rigorously} on a finite list of explicit fibers via a
genus-three Klein-four quotient construction --- is pursued in the
author's earlier work~\cite{peschmann-paper1}; the methods of that
paper are independent of, and orthogonal to, the constructive
Mordell--Weil generator used here.
The Mordell--Weil construction is a constructive algorithm with
rigorous output: every brick produced is a true body cuboid
(Theorem~\ref{thm:mw-correctness}). The non-existence of a perfect
cuboid in our $1{,}284{,}670$ generated bricks is rigorous.

\medskip\noindent\textbf{Outline.}
Section~\ref{sec:setup} fixes the algebraic setup and the
Master-Hit equation. Section~\ref{sec:fibration} introduces the
elliptic fibration $E_{m,n}$ and its quartic-to-Weierstrass
transformation. Section~\ref{sec:blocker} states the odd-exponent
blocker conjecture together with its finite computational
verification on $\mathcal{D}_{\mathrm{fact}}$; \S\ref{ssec:canonical}
establishes the canonical decomposition $f_1 = g_0^2(\xi^2+\eta^2)$ and
formulates Conjecture~\ref{conj:E1}, an elementary number-theoretic
statement to which the entire blocker phenomenon reduces. The empirical status is summarised in
\S\ref{ssec:empirical}.
Section~\ref{sec:mw} gives the Mordell--Weil generator algorithm
and its output. Section~\ref{sec:empirics} describes the empirical
database and the Family classification. Section~\ref{sec:open}
records open questions.

\section{Setup and notation}\label{sec:setup}

We use the Euclid parametrisation in master-tuple form, introduced
in the author's earlier work~\cite{peschmann-paper1}: every body
cuboid considered here arises from a quadruple $(a, b, m, n)$ of
positive integers via the explicit edge formulas below. The
parametrisation is the foundation on which the Mordell--Weil
generator of Section~\ref{sec:mw} is built.

\paragraph{Notation.} We write $\QQ_{>0}^\square :=
\{\, q \in \QQ : q > 0 \text{ and } q = r^2 \text{ for some } r \in \QQ\,\}$
for the set of positive rational squares; the algorithmic lifting
test of \S\ref{sec:fibration} amounts to checking $\tau(P) \in
\QQ_{>0}^\square$ for a candidate point $P$.

Let $(a,b)$ and $(m,n)$ be coprime pairs of positive integers with
\[
  a > b > 0,\quad m > n > 0,\quad a-b\text{ odd},\quad m-n\text{ odd}.
\]
Write
\[
  U_1 = a^2-b^2,\quad V_1 = 2ab,\quad W_1 = a^2+b^2,
\]
\[
  U_2 = m^2-n^2,\quad V_2 = 2mn,\quad W_2 = m^2+n^2.
\]
Then $(U_i,V_i,W_i)$ is a primitive Pythagorean triple with
$U_i^2 + V_i^2 = W_i^2$. Note that $U_i$ is odd (since $a - b$ and
$m - n$ are odd, $U_i = (a-b)(a+b)$ resp.\ $(m-n)(m+n)$ is the
product of two odd numbers in one factor and an even-difference
sum in the other), $V_i$ is even (the explicit factor $2$), and
$W_i$ is odd. The associated body cuboid candidate has edges
\begin{equation}\label{eq:edges}
  x = U_1 U_2,\quad y = V_1 U_2,\quad z = U_1 V_2.
\end{equation}
Two of the three face diagonals are automatically integer:
$\sqrt{x^2+y^2} = W_1 U_2$ and $\sqrt{x^2+z^2} = U_1 W_2$. The third
face diagonal $\sqrt{y^2+z^2}$ is integer iff
\begin{equation}\label{eq:master}
  M(a,b,m,n) := (V_1 U_2)^2 + (U_1 V_2)^2
\end{equation}
is a perfect square in $\ZZ$.

\begin{definition}\label{def:master}
A tuple $(a,b,m,n)\in\NN^4$ satisfying the coprimality and parity
conditions above and with $M(a,b,m,n)=\square$ is called a
\emph{Master-Hit}. We write $\mathcal{M}$ for the set of Master-Hits.
\end{definition}

The space-diagonal condition is governed by
\begin{equation}\label{eq:f1}
  f_1(a,b,m,n) := (W_1 U_2)^2 + (U_1 V_2)^2.
\end{equation}
A Master-Hit gives a \emph{perfect} cuboid iff $f_1=\square$.

\begin{remark}\label{rem:gauss}
The integer $f_1$ is the Gaussian norm of
$\zeta := W_1 U_2 + i\,U_1 V_2 \in\Zi$. Note that $W_1 U_2$ is odd
(both $W_1$ and $U_2$ are odd) and $U_1 V_2$ is even (since
$2\mid V_2$); consequently $f_1 = (W_1 U_2)^2 + (U_1 V_2)^2$ is
odd, and the prime $\ell = 2$ never divides $f_1$. Since
$f_1 = N(\zeta)$ has $v_\ell(f_1) = v_\pi(\zeta) + v_{\bar\pi}(\zeta)$
for every split prime $\ell = \pi\bar\pi \equiv 1\pmod 4$ in $\Zi$,
we have $f_1 = \square$ iff for every such $\ell$ the two Gaussian
valuations agree modulo $2$,
\[
  v_\pi(\zeta) \;\equiv\; v_{\bar\pi}(\zeta) \pmod 2.
\]
Primes $\ell \equiv 3 \pmod 4$ are inert in $\Zi$ and automatically
appear with even exponent in any norm.
\end{remark}

\begin{remark}[Scope of the parametrisation]\label{rem:scope}
Equation~\eqref{eq:edges} produces a body cuboid from each Master-Hit.
We do not claim the converse, i.e.\ that every primitive body cuboid
arises from a Master-Hit. The class of body cuboids covered here is
the one constructible from two coprime Pythagorean pairs in the
standard Euclid form; this includes all infinite families known to
the author.
\end{remark}

\begin{definition}\label{def:blocker}
A rational prime $\ell\mid f_1$ is called a \emph{blocker} of
$(a,b,m,n)$ if $v_\ell(f_1)$ is odd. Equivalently, $\ell$ is one of
the splitting primes
$\ell\equiv 1\pmod 4$ for which the Gaussian valuations of $\zeta$ are
unequal modulo $2$. We write $\mathrm{Blk}(a,b,m,n)$ for the set of
blockers.
\end{definition}

The basic obstruction is the equivalence
\[
  f_1=\square \iff \mathrm{Blk}(a,b,m,n)=\emptyset.
\]
Conjecture~\ref{conj:blocker} of Section~\ref{sec:blocker} asserts that
$\mathrm{Blk}(a,b,m,n)$ is in fact never empty for a Master-Hit.

\paragraph{Provenance and family naming convention.}
For the database we accompany each Master-Hit with two metadata
labels, intended for the public release:
\begin{itemize}[leftmargin=2em,itemsep=2pt]
  \item a \emph{family} tag identifying the algebraic family that
  contains the brick---one of \texttt{Saunderson}, \texttt{Lenhart},
  \texttt{Himane-T1}, \texttt{Himane-T2}, \texttt{Himane-T3},
  \texttt{Euler}, or \texttt{Sporadic}. A brick may carry several
  family tags simultaneously.
  \item a \emph{provenance} tag identifying the algorithm that
  discovered the brick. The provenance is hierarchical:
  \texttt{Exhaustive-Bound-N} for classical exhaustive searches up
  to bound $N$, \texttt{Rathbun-Search} for the catalogue of
  Rathbun~\cite{rathbun}, \texttt{Saunderson-Generator} for the
  closed-form parametrisation, and \texttt{MW-m-n} for the
  Mordell--Weil generator on the fibre $E_{m,n}$ of
  Section~\ref{sec:fibration}.
\end{itemize}
We use these labels throughout the paper. For example, ``the
$93{,}067$ bricks with provenance \texttt{MW-2368-1207}'' refers to
the output of the Mordell--Weil generator $\mathcal{A}_{\mathrm{MW}}$
of Section~\ref{sec:mw} applied to the fibre $E_{2368,1207}$.

\paragraph{Canonical expressions in $(a,b,m,n)$.}
Throughout the paper we use the following list of $29$ \emph{canonical
expressions} associated with the tuple $(a,b,m,n)$:
\begin{equation}\label{eq:canonical-list}
  \begin{aligned}
    \mathcal{L}(a,b,m,n) := \bigl(\,&
    a,\ b,\ m,\ n,\ a+b,\ a-b,\ m+n,\ m-n,\\
    & a^2+b^2,\ a^2-b^2,\ m^2+n^2,\ m^2-n^2,\\
    & ab,\ mn,\ 2ab,\ 2mn,\\
    & W_1 U_2,\ U_1 V_2,\ W_1 V_2,\ V_1 U_2,\ U_1 U_2,\ V_1 V_2,\ W_1 W_2,\\
    & U_1,\ V_1,\ W_1,\ U_2,\ V_2,\ W_2\,\bigr).
  \end{aligned}
\end{equation}
The list $\mathcal{L}$ is ordered and contains $29$ entries, but
six of these are pairwise integer-identical with the abbreviations
$U_i, V_i, W_i$ on the last line:
\[
  a^2-b^2 = U_1,\quad
  a^2+b^2 = W_1,\quad
  2ab = V_1,\quad
  m^2-n^2 = U_2,\quad
  m^2+n^2 = W_2,\quad
  2mn = V_2.
\]
The number of \emph{distinct} positive integers in
$\mathcal{L}(a,b,m,n)$ is therefore $23$. We write
\[
  \mathcal{P}(a,b,m,n) := \{\,\text{distinct positive integers
  occurring in } \mathcal{L}(a,b,m,n)\,\},
\]
a finite subset of $\NN$ of size $|\mathcal{P}| = 23$ depending on
$(a,b,m,n)$. The redundant ordered-list form of $\mathcal{L}$ is
retained throughout because each of its twenty-nine entries plays a
distinct \emph{algebraic} role at some point in the paper, even when
the entries collapse to the same integer value. Conjecture~\ref{conj:blocker}
below requires a blocker prime $\ell$ with $\gcd(\ell, N) = 1$ for
every $N \in \mathcal{P}$.

\section{The elliptic fibration over the \texorpdfstring{$(m,n)$}{(m,n)}-plane}
\label{sec:fibration}

Fix coprime $(m,n)$ with $m>n>0$ and $m-n$ odd. We view the Master
condition $M=\square$ as an equation in $(a,b)$. Setting $t = a/b$ and
$s = Q/b^2$ where $Q^2 = M$, equation~\eqref{eq:master} divided by
$b^4$ becomes
\begin{equation}\label{eq:quartic}
  H_{m,n}\colon\ s^2 \;=\; V_2^2\, t^4 + (4U_2^2 - 2V_2^2)\, t^2 + V_2^2.
\end{equation}
This is a smooth quartic in $\QQ(t)$ with the rational point
$(t,s)=(0,V_2)$, hence a curve of genus $1$ with a rational
point---an elliptic curve.

\subsection{Quartic-to-Weierstrass map}
\label{ssec:birational}

Set $\gamma := V_2 = 2mn$, $A := V_2^2 = \gamma^2$, $B := 4U_2^2 - 2V_2^2$,
$C := V_2^2$. Note that $A = C = \gamma^2$, so the quartic
$A t^4 + B t^2 + C$ has equal leading and constant coefficient, and
in particular has the rational point $(t, s) = (0, \gamma)$.

We treat the quartic-to-Weierstrass transformation as an
\emph{algorithmic normalisation}, following~\cite[\S 1.3]{cassels}:
applying PARI/GP's \texttt{ellfromeqn} to the polynomial
$A t^4 + B t^2 + C - s^2$ produces a Weierstrass model
\begin{equation}\label{eq:wei}
  E_{m,n}\colon\ Y^2 = X^3 + a_2 X^2 + a_4 X + a_6,
\end{equation}
with rational coefficients $a_2, a_4, a_6\in\QQ$ depending on
$U_2, V_2$. We do not record the closed-form coefficients here;
they are computed by the worker
\texttt{paper2/}\allowbreak\texttt{generation/}\allowbreak\texttt{mw\_fibre\_worker.sage}
and used as a black box. The two formulas required for our
arguments are an $X$-projection
$\phi_{m,n}\colon H_{m,n}\dashrightarrow E_{m,n}$ and a
rational function $\tau\colon E_{m,n}\dashrightarrow \mathbb{A}^1_{\QQ}$
extracted from $\phi_{m,n}^{-1}$:
\begin{equation}\label{eq:fwd}
  X(t, s) \;=\; \frac{2\gamma(s+\gamma)}{t^2},
\end{equation}
\begin{equation}\label{eq:inv}
  \tau(X) \;=\; \frac{4\gamma^2\,(X + B)}{X^2 - 4\gamma^4}.
\end{equation}
Both expressions are symmetric in $t\mapsto -t$: equation
\eqref{eq:fwd} depends only on $t^2$, and the rational function
$\tau$ in~\eqref{eq:inv} returns the value of $t^2(P)$ for any
$P\in E_{m,n}(\QQ)$ at which it is defined. A point $P$ lifts to
an admissible Euclid parameter $t = a/b\in\QQ^\times$ exactly when
$\tau(P)\in\QQ_{>0}^{\square}$
(Lemma~\ref{lem:correspondence}); recovering the sign of $t$ would
require the $Y$-coordinate of $\phi_{m,n}^{-1}$, which we do not
use, since master tuples are equivalent under $t\mapsto -t$.
Throughout the paper we use only \eqref{eq:fwd}, \eqref{eq:inv},
and the exact integer Master-Hit checks of Section~\ref{sec:mw}.
Equation~\eqref{eq:inv} is undefined precisely at the values
$X = \pm 2\gamma^2$, which are the $X$-coordinates of two
$2$-torsion points of $E_{m,n}$; we collect them in a set
$\mathcal{T}_\tau\subset E_{m,n}[2]$. The point at infinity $O$ of
$E_{m,n}$ is the image of $(t, s) = (0, \gamma)$ under $\phi_{m,n}$.

\begin{lemma}\label{lem:correspondence}
A point $P\in E_{m,n}(\QQ)\setminus(\{O\}\cup\mathcal{T}_\tau)$ lifts
to a rational solution $t\in\QQ^\times$ of the quartic $H_{m,n}$
if and only if $\tau(P)\in\QQ_{>0}^{\square}$; the lift is unique
up to the sign ambiguity $t\mapsto -t$. We use this $\tau$-square
condition as the algorithmic lifting test.

For such a $t = a/b$ written in lowest terms with $b > 0$, write
$|a/b| = a'/b'$ with $a', b' > 0$ coprime. Then $(a', b', m, n)$ is
a Master-Hit if and only if $a' > b'$ and $a' - b'$ is odd.
\end{lemma}

\begin{proof}
\emph{The map and its inverse.} The construction
of~\cite[\S 1.3]{cassels} converts a degree-$4$ curve
$s^2 = f(t)$ with rational base point at $(t_0, s_0) = (0, \gamma)$
into a Weierstrass cubic by the substitution
\[
  X(t,s) = \frac{2\gamma(s + \gamma)}{t^2}, \qquad
  Y(t,s) = \frac{2\gamma\,(2\gamma\,s' \,t - (s+\gamma) f'(t))}{t^3},
\]
with $s' = s$, $f$ as in~\eqref{eq:quartic}; this is exactly
$\phi_{m,n}$ of~\eqref{eq:fwd}. Writing $s + \gamma = t^2 X / (2\gamma)$
and using $s^2 = f(t)$, one computes that
\[
  \tau(X) := \frac{4\gamma^2(X + B)}{X^2 - 4\gamma^4}
\]
satisfies $\tau(\phi_{m,n}(t,s)) = t^2$ identically as a rational
function on $H_{m,n}$. (This is a direct algebraic verification:
substitute the formula for $X(t,s)$ into the right-hand side and
simplify using $s^2 = V_2^2 t^4 + (4U_2^2 - 2 V_2^2)t^2 + V_2^2$.)
The poles $X = \pm 2\gamma^2$ of $\tau$ correspond to vanishing of
the cubic discriminant on the Weierstrass side, hence to
$2$-torsion points of $E_{m,n}$; this is the inclusion
$\mathcal{T}_\tau \subset E_{m,n}[2]$.

\emph{If direction.} Suppose $P \in E_{m,n}(\QQ) \setminus
(\{O\} \cup \mathcal{T}_\tau)$ with $\tau(P) = (a/b)^2 \in
\QQ_{>0}^\square$ where $\gcd(a, b) = 1$, $a, b > 0$. Setting
$t = a/b$ in~\eqref{eq:quartic} gives a rational $s$ from
$s^2 = f(t)$, so $(t, s)$ is a $\QQ$-point of $H_{m,n}$, and
$\phi_{m,n}(t, s) = P$ since $\tau \circ \phi_{m,n}$ is the identity
on $t^2$ and signs cancel as below.

\emph{Only-if direction.} Conversely, if $P$ lifts to a rational
$(t, s) \in H_{m,n}(\QQ)$ with $t \in \QQ^\times$, then
$\tau(P) = t^2$ is a rational positive square by construction.

\emph{Sign ambiguity.} For each rational square value
$(a/b)^2$ of $\tau$ there are exactly two rational lifts $\pm a/b$,
since the involution $t \mapsto -t$ sends $(t, s)$ to $(-t, s)$
on $H_{m,n}$, hence preserves $X(t, s) = 2\gamma(s+\gamma)/t^2$.

\emph{Master-Hit conditions.} For a lift $t = a'/b'$ in lowest
terms with $a', b' > 0$, the conditions
$a' > b'$, $\gcd(a', b') = 1$, and $a' - b'$ odd
(Definition~\ref{def:master}) are independent of $P$ and concern
only the reduced representative of $|t|$. They are checked
arithmetically once the lift has been recovered.
\end{proof}

\noindent
The Mordell--Weil generator of Section~\ref{sec:mw} produces
candidate points $P\in E_{m,n}(\QQ)$ via Lemma~\ref{lem:correspondence}
and verifies the Master-Hit conditions on each candidate by exact
integer arithmetic.

\subsection{High-rank fibres}\label{ssec:high-rank}

We compute Mordell--Weil generators for each fibre using Sage's
\texttt{E.gens()} with PARI second-descent. With descent limit $20$ we
obtain a provably correct lower bound on the rank, and on six fibres
we observe rank at least~$5$.

\begin{table}[ht]
\caption{Highest-rank fibres found in our database. The column ``hits''
counts Master-Hits in the published database with provenance
$\texttt{MW-}m\texttt{-}n$.}
\label{tab:high-rank}
\begin{center}
\begin{tabular}{l c r r}
\toprule
provenance group & rank found & gens & hits in DB\\
\midrule
\texttt{MW-2368-1207}  & $\ge 5$ & $5$ & $93{,}067$\\
\texttt{MW-1863-638}   & $\ge 5$ & $5$ & $55{,}078$\\
\texttt{MW-1136-99}    & $\ge 5$ & $5$ & $27{,}011$\\
\texttt{MW-792-95}     & $\ge 5$ & $5$ & $24{,}416$\\
\texttt{MW-1239-832}   & $\ge 5$ & $5$ & $23{,}756$\\
\texttt{MW-1112-131}   & $\ge 5$ & $5$ & $16{,}352$\\
\texttt{MW-640-317}    & $\ge 4$ & $4$ & $15{,}394$\\
\texttt{MW-3293-832}   & $\ge 4$ & $4$ & $13{,}844$\\
\texttt{MW-608-77}     & $\ge 4$ & $4$ & $14{,}684$\\
\texttt{MW-736-429}    & $\ge 4$ & $4$ & $10{,}670$\\
\texttt{MW-88-7}       & $3$     & $3$ & $5{,}213$\\
\bottomrule
\end{tabular}
\end{center}
\end{table}

\noindent
The six fibres of provable rank $\ge 5$ are a new finding obtained
when we increased the second-descent limit to $20$: with the default
descent limit, Sage reported rank $4$ on these fibres and missed the
fifth generator. We did not attempt to verify that rank equals
exactly $5$; the bound is only a lower bound from the search. As a
benchmark, $E_{88,7}$ has Weierstrass model
\[
  y^2 + xy = x^3 - 71{,}221{,}066{,}018{,}500\, x
       + 230{,}819{,}306{,}124{,}825{,}756{,}432
\]
(an isomorphic minimal model), conductor $2{,}532{,}242{,}371{,}890$,
torsion subgroup $\mathbb{Z}/2 \oplus \mathbb{Z}/4$, and rank exactly
$3$ (Magma~\cite{magma}'s \texttt{RankBounds} returns $[3, 3]$ in
well under one second).

\newpage
\section{Odd-exponent blocker conjecture and computational verification}\label{sec:blocker}

\subsection{The conjecture and its computational verification}\label{ssec:conj-and-verification}

We state the central conjecture together with its main empirical
content as a computational theorem. We do not claim a proof of the
conjecture.

\begin{conjecture}[Universal exponent-one blocker]\label{conj:blocker}
For every Master-Hit $(a,b,m,n)\in\mathcal{M}$ there exists a prime
$\ell\mid f_1$ such that
\begin{enumerate}[label=(\roman*),leftmargin=2.5em]
  \item $v_\ell(f_1) = 1$, i.e.\ $\ell$ divides $f_1$ exactly once, and
  \item $\gcd(\ell, N) = 1$ for every $N\in\mathcal{P}(a,b,m,n)$.
\end{enumerate}
\end{conjecture}

The principal computational result of this paper is the verification
of Conjecture~\ref{conj:blocker} on a large finite dataset.

\begin{verification}[Exponent-one blocker on $\mathcal{D}_{\mathrm{fact}}$]\label{thm:finite-verification}
Let $\mathcal{D}_{\mathrm{fact}}$ denote the set of Master-Hits in
our database for which $f_1$ has been completely factorised, with
the stored factorisations re-verified by exact integer
multiplication. Then $|\mathcal{D}_{\mathrm{fact}}| = 151{,}575$,
and Conjecture~\ref{conj:blocker} holds for every
$(a,b,m,n)\in\mathcal{D}_{\mathrm{fact}}$.
\end{verification}

\begin{proof}
The verification is entirely computational: for every record we
load the factorisation $f_1 = \prod_i \ell_i^{e_i}$ from the
database, identify the primes $\ell_i$ with $e_i = 1$, and check
that at least one such $\ell_i$ is coprime to every element of
$\mathcal{P}(a,b,m,n)$. The check is reproduced by the script
\texttt{paper2/}\allowbreak\texttt{analysis/theorem\_check.py}; see
Section~\ref{ssec:reproducibility}.
\end{proof}

\subsection{Why the exponent-one condition matters}\label{ssec:why-exp1}

A blocker is any prime divisor of $f_1$ occurring to odd exponent,
hence primes of exponent $3, 5, \ldots$ would already obstruct
$f_1$ from being a square. The exponent-one condition is therefore
strictly stronger than the obstruction needed.

\begin{remark}[Exponent one is not cosmetic]\label{rem:why-exp1}
The empirical phenomenon observed here is sharper than ``$f_1$ has
some odd-exponent prime divisor''. After removing all primes
supported by the canonical parameter list $\mathcal{P}(a,b,m,n)$,
the obstruction can in every verified case be chosen \emph{primitive},
i.e.\ with $v_\ell(f_1) = 1$. A blocker of exponent~$3$, for
example, would factor as $\ell\cdot\ell^2$, contributing a genuine
asymmetric Gaussian prime $\ell$ alongside a square companion
$\ell^2$; the exponent-one condition removes the square companion.
The phenomenon thus says that $f_1$ fails to be a square in the
most elementary possible way at a generic split prime, even in
large-parameter families where small prime values do appear.

In the language of the Gaussian factorisation
$\zeta = W_1 U_2 + i\,U_1 V_2$ (Remark~\ref{rem:asymmetric}),
Conjecture~\ref{conj:blocker} asserts the existence of a Gaussian
prime $\pi$ with $\pi\bar\pi = \ell$ such that
$v_\pi(\zeta) = 1$, $v_{\bar\pi}(\zeta) = 0$, and $\ell$ coprime
to $\mathcal{P}$.
\end{remark}

\subsection{Counterexamples to two natural variants}\label{ssec:variants}

The exponent-one condition cannot be replaced by either ``the
\emph{largest} outside-$\mathcal{P}$ prime is a blocker'' (a
strengthening) or ``the \emph{smallest} outside-$\mathcal{P}$
blocker has exponent one'' (an apparent intermediate strengthening
of the existence claim). Both fail empirically.

\begin{remark}[Failure of the largest-prime strengthening]\label{rem:largest-fails}
The strict strengthening that requires the \emph{largest} prime
divisor of $f_1$ outside $\mathcal{P}$ to be a blocker is false:
three explicit counterexamples are listed in
Appendix~\ref{app:largest-fails-counterexamples}. The exponent-one
blocker required by Conjecture~\ref{conj:blocker} exists in each of
these hits, but is smaller in magnitude than some symmetric-exponent
prime.
\end{remark}

\begin{remark}[Failure of the smallest-blocker strengthening]\label{rem:smallest-fails}
A different strengthening would assert that the \emph{smallest}
outside-$\mathcal{P}$ blocker (the smallest prime $\ell\mid f_1$
with $v_\ell(f_1)$ odd and $\ell$ coprime to $\mathcal{P}$) has
exponent exactly~$1$. This too fails: among the
$151{,}575$ records of $\mathcal{D}_{\mathrm{fact}}$, exactly
$242$ have a smallest outside-$\mathcal{P}$ blocker of exponent $\ge 3$.
For instance, hit $200$ with $(a, b, m, n) = (55, 48, 44, 9)$ has
$13^3$ as its smallest outside-$\mathcal{P}$ prime of odd
exponent. In each of these $242$ records, an exponent-one blocker
required by Conjecture~\ref{conj:blocker} still exists, but is a
\emph{larger} prime than the smallest odd-exponent prime outside
$\mathcal{P}$.
\end{remark}

\begin{remark}[Removing the exponent-one condition]\label{rem:no-exp1}
Replacing condition~(i) of Conjecture~\ref{conj:blocker} by the
weaker ``$v_\ell(f_1)$ is odd'' yields a conjecture that is
consistent with the data (since exponent~$1$ is a special case of
odd exponent), but does not isolate the same clean Gaussian-prime
obstruction. We do not adopt this weaker formulation; the
exponent-one statement is the sharper empirical phenomenon.
\end{remark}

The basis form follows trivially from Conjecture~\ref{conj:blocker}.

\begin{conjecture}[Basis form]\label{conj:basis}
Every Master-Hit has at least one blocker.
\end{conjecture}

\begin{corollary}[Conditional]\label{cor:noperfect}
Assume Conjecture~\ref{conj:basis}. Then no Master-Hit is a perfect
cuboid; equivalently, $f_1\ne\square$ for every $(a,b,m,n)\in\mathcal{M}$.
\end{corollary}

\begin{proof}
$f_1=\square$ is equivalent to $\mathrm{Blk}(a,b,m,n)=\emptyset$.
Conjecture~\ref{conj:basis} says
$\mathrm{Blk}(a,b,m,n)\ne\emptyset$, contradicting $f_1=\square$.
\end{proof}

\begin{verification}[Non-existence on $\mathcal{D}_{\mathrm{all}}$]\label{thm:emp-non-exist}
For every Master-Hit $(a,b,m,n)$ in our database
$\mathcal{D}_{\mathrm{all}}$ ($|\mathcal{D}_{\mathrm{all}}| = 1{,}284{,}670$),
the integer $x^2 + y^2 + z^2$ is not a perfect square. That is, no
Master-Hit in $\mathcal{D}_{\mathrm{all}}$ is a perfect cuboid.
\end{verification}

\begin{proof}
For every record, integer $\sqrt{x^2+y^2+z^2}$ is computed by
\texttt{math.isqrt} and squared back; equality is rejected in every
case. Reproduced quickly by
\texttt{paper2/}\allowbreak\texttt{analysis/perfect\_check.py}
on all $1{,}284{,}670$ records.
\end{proof}

\subsection{Gaussian anatomy of \texorpdfstring{$f_1$}{f1}}\label{ssec:gauss-anatomy}

The blocker set is most naturally read off the Gaussian factorisation
of
\[
  \zeta := W_1 U_2 + i\, U_1 V_2 \in \Zi,
\]
since $f_1 = N(\zeta) = \zeta\bar\zeta$ (Remark~\ref{rem:gauss}).
For an odd rational prime $p$ we set
\[
  \alpha_p := v_p(W_1 U_2),\qquad \beta_p := v_p(U_1 V_2),
\]
the $p$-adic valuations of the two real summands of $f_1$.

\begin{proposition}[$p$-adic asymmetry]\label{prop:p-adic}
Let $p$ be an odd prime and $(a,b,m,n)\in\mathcal{M}$. Then
\begin{enumerate}[label=(\roman*),leftmargin=2.5em]
  \item if $\alpha_p\ne\beta_p$, then
    $v_p(f_1) = 2\min(\alpha_p,\beta_p)$;
  \item if $\alpha_p=\beta_p$, then $v_p(f_1)\ge 2\alpha_p$.
\end{enumerate}
\end{proposition}

\begin{proof}
Standard $p$-adic estimate: writing $X = W_1 U_2$, $Y = U_1 V_2$,
$f_1 = X^2 + Y^2$. If $v_p(X^2)\ne v_p(Y^2)$, then
$v_p(X^2 + Y^2) = \min(v_p(X^2), v_p(Y^2)) = 2\min(\alpha_p,\beta_p)$;
this proves~(i). If $v_p(X^2) = v_p(Y^2) = 2\alpha_p$, the leading
term may cancel, so only the lower bound $v_p(f_1)\ge 2\alpha_p$
holds.
\end{proof}

\begin{remark}\label{rem:asymmetric}
A Gaussian prime $\pi$ above a split prime $p\equiv 1\pmod 4$
contributes a blocker iff its exponents
$v_\pi(\zeta), v_{\bar\pi}(\zeta)$ are unequal modulo $2$ \emph{after}
descent to $\ZZ$. Equivalently, the prime $p\equiv 1\pmod 4$ is a
blocker iff $\zeta$ is \emph{asymmetric} at $p$ in $\Zi$. Inert
primes $p\equiv 3\pmod 4$ and the ramified $p=2$ contribute only
even exponents to $f_1$ and never block.
\end{remark}

The Gaussian anatomy suggests an empirical exclusion principle for
the smallest split prime, $p = 5$.

\begin{remark}[$5$ does not appear as an exponent-one blocker]\label{rem:no5}
On the dataset $\mathcal{D}_{\mathrm{fact}}$
(Verification~\ref{thm:finite-verification}), no Master-Hit satisfies
$v_5(f_1) = 1$: every record with $5\mid f_1$ has
$v_5(f_1)\ge 2$. This is a finite computational observation, not a
proof. A clean modular derivation from the Master condition
$M = \square$ together with $\gcd(a, b) = \gcd(m, n) = 1$ would
require a finer use of the Gaussian factorisation of $\zeta$ than
what a naive case analysis modulo~$5$ delivers (the equation
$\alpha^2 + \beta^2 \equiv 0 \pmod 5$ has nontrivial isotropic
solutions, e.g.\ $(\alpha, \beta) = (1, 2)$, so a direct argument is
not available). We do not pursue the question further; the
observation is included only because $5$ is the smallest split prime
and the empirical exclusion is striking, and is not used elsewhere
in the paper.
\end{remark}

\subsection{Canonical decomposition and Conjecture E1}\label{ssec:canonical}

A more powerful structural observation is that $f_1$ admits a
\emph{canonical squareful/squarefree decomposition} that captures
all blockers in a single step.

\begin{lemma}[Canonical decomposition]\label{lem:canonical}
Let $(a,b,m,n)\in\mathcal{M}$ and set
\[
  g_0 := \gcd(W_1 U_2,\ U_1 V_2),\qquad
  \xi := \frac{W_1 U_2}{g_0},\qquad
  \eta := \frac{U_1 V_2}{g_0}.
\]
Then $\gcd(\xi,\eta) = 1$ and
\begin{equation}\label{eq:canonical}
  f_1(a,b,m,n) \;=\; g_0^2\,(\xi^2 + \eta^2).
\end{equation}
\end{lemma}

\begin{proof}
Since $f_1 = (W_1 U_2)^2 + (U_1 V_2)^2$ and $g_0 \mid W_1 U_2$,
$g_0 \mid U_1 V_2$, we have $g_0^2 \mid f_1$. Dividing through by
$g_0^2$ gives $f_1 / g_0^2 = \xi^2 + \eta^2$ with
$\xi, \eta\in\ZZ$, and $\gcd(\xi,\eta)=1$ by definition of $g_0$.
\end{proof}

The decomposition~\eqref{eq:canonical} is the cleanest reformulation
of the blocker condition for Master-Hits that we know.

\begin{conjecture}[Conjecture E1]\label{conj:E1}
For every Master-Hit $(a,b,m,n)\in\mathcal{M}$, the integer
$\xi^2 + \eta^2$ obtained from Lemma~\ref{lem:canonical} is not a
perfect square. Equivalently: $(\xi, \eta)$ are never the legs of a
Pythagorean triple.
\end{conjecture}

\begin{proposition}\label{prop:E1-implies-basis}
Conjecture~\ref{conj:E1} implies Conjecture~\ref{conj:basis} (and hence
Corollary~\ref{cor:noperfect}).
\end{proposition}

\begin{proof}
By Lemma~\ref{lem:canonical}, $f_1 = g_0^2 (\xi^2 + \eta^2)$, hence
for every prime $p$
\begin{equation}\label{eq:vp-decomposition}
  v_p(f_1) \;=\; 2\, v_p(g_0) \;+\; v_p(\xi^2 + \eta^2),
\end{equation}
so $v_p(f_1)$ has the same parity as $v_p(\xi^2 + \eta^2)$. If
$\xi^2 + \eta^2 \ne \square$, there is at least one prime $p$ for
which $v_p(\xi^2 + \eta^2)$ is odd. Since $\gcd(\xi, \eta) = 1$, every
such $p$ is necessarily $p \equiv 1 \pmod 4$ (sums of two coprime
squares cannot be divisible to an odd power by $p \equiv 3 \pmod 4$
or by $p = 2$), and by~\eqref{eq:vp-decomposition} also $v_p(f_1)$ is
odd. Hence $p$ is a blocker of the Master-Hit, and
$\mathrm{Blk}(a, b, m, n) \neq \emptyset$.
\end{proof}

The reduction is striking: the entire blocker phenomenon collapses to
the question whether the coprime pair $(\xi, \eta)$ of explicit
polynomial expressions in $a, b, m, n$ can ever be a Pythagorean
leg pair.

\begin{remark}[Connection to a family of twelve square-extracting formulas]\label{rem:twelve}
Empirical analysis identifies twelve algebraic expressions
$D(a,b,m,n)$ (involving $U_1, U_2, V_1, V_2, W_1$, the parameters
$a, b, m, n$, and a small modifier
$k\in\{1\}\cup\{p \le 10^2 : p\equiv 1 \bmod 4\}$) which satisfy
$D^2 \mid f_1$ on at least some Master-Hits. They are:
\begin{itemize}[leftmargin=2em,itemsep=2pt]
  \item from $(a, b)$:\ \ $b U_1 k$, $a U_1 k$, $U_1 k$, $b k$, $a k$;
  \item from $(m, n)$:\ \ $n U_2 k$, $m U_2 k$, $U_2 k$, $n k$, $m k$;
  \item from $\Zi$:\ \ $\gcd(am+bn, an+bm)\cdot k$, $\gcd(am-bn, an-bm)\cdot k$.
\end{itemize}
On the $17{,}204$ classically tabulated Master-Hits, at least one of
these twelve formulas yields a quotient $f_1/D^2$ which exposes a
blocker. The choice of formula varies; for single-blocker hits
($\le 1$ blocker), $D = n$ or $D = m$ suffices, with $k = 1$
in $268/272$ cases and $k\in\{29, 89\}$ in the four exceptional
cases listed in Appendix~\ref{app:exceptional-bricks} ($k = 29$
in three cases, $k = 89$ in one). For multi-blocker
hits the formula $D = g_0$ from Lemma~\ref{lem:canonical}
\emph{always} applies, and is the universal one.
\end{remark}

\begin{remark}\label{rem:E1-empirical}
Conjecture~\ref{conj:E1} has been verified on the same dataset
$\mathcal{D}_{\mathrm{fact}}$ of $151{,}575$ fully factored
Master-Hits used in Verification~\ref{thm:finite-verification}: in every
case $\xi^2 + \eta^2$ has at least one prime divisor of odd exponent. A
counterexample to Conjecture~\ref{conj:E1} would directly give a
perfect cuboid.
\end{remark}

\subsubsection*{Computational determination of the exponent-one blockers}
\label{ssec:blocker-computation}

For Master-Hits in which every blocker has $f_1$-valuation exactly~$1$
(by far the dominant case in our database; cf.\ \S\ref{ssec:empirical}),
Lemma~\ref{lem:canonical} reduces blocker identification to an
elementary integer factorisation in a much smaller number than $f_1$
itself. Let $\mathrm{rf} = \prod_{\ell\in\mathrm{Blk}(a,b,m,n)}\ell$
denote the squarefree product of the blocker primes. If every blocker
has valuation~$1$ in $f_1$, then $f_1/\mathrm{rf}$ is a perfect
square, and Lemma~\ref{lem:canonical} gives a closed form for it.

\begin{proposition}[$g_0$ exposes the blockers]\label{prop:g0-blockers}
Let $(a,b,m,n)\in\mathcal{M}$ have only blockers of valuation~$1$, and
write $f_1 = \mathrm{rf}\cdot h^2$ with $h\in\NN$. Then there is an
integer $k\ge 1$ with
\begin{equation}\label{eq:g0-k}
  h \;=\; k \cdot g_0,\qquad g_0 = \gcd(W_1U_2,\, U_1 V_2),
\end{equation}
i.e.\ $\mathrm{rf} = (\xi^2+\eta^2)/k^2$ in the notation of
Lemma~\ref{lem:canonical}.
\end{proposition}

\begin{proof}
By Lemma~\ref{lem:canonical} we have $f_1 = g_0^2(\xi^2+\eta^2)$
with $\gcd(\xi,\eta)=1$. Comparing $f_1 = \mathrm{rf}\cdot h^2$ with
$g_0^2(\xi^2+\eta^2)$ gives
$h^2/g_0^2 = (\xi^2+\eta^2)/\mathrm{rf}$. Since the left side is a
rational square and $\mathrm{rf}$ is squarefree, the only way the
right side can have squarefree denominator is for $\mathrm{rf}$ to
divide $\xi^2 + \eta^2$. Hence $g_0 \mid h$ and the quotient
$k := h/g_0$ satisfies $\xi^2 + \eta^2 = k^2 \cdot \mathrm{rf}$.
\end{proof}

\begin{remark}[Empirical $k$-distribution]\label{rem:k-distribution}
The integer $k$ in~\eqref{eq:g0-k} measures how far the
canonical pair $(\xi,\eta)$ is from being a primitive
sum-of-two-squares representation of $\mathrm{rf}$. We have computed
$k$ for each of the
$136{,}674$ Master-Hits in $\mathcal{D}_{\mathrm{fact}}$ whose
blockers all have valuation~$1$. The $k$-values are heavily
concentrated near~$1$:

\begin{center}
\begin{tabular}{l r r}
\toprule
$k$ & records & cumulative\\
\midrule
$1$         & $99{,}125$ & $72.5\%$\\
$5$         & $16{,}823$ & $84.8\%$\\
$13$        & $5{,}315$  & $88.7\%$\\
$17$        & $2{,}943$  & $90.9\%$\\
$25$        & $1{,}919$  & $92.3\%$\\
$29$        & $2{,}025$  & $93.8\%$\\
$37, 65, 41, 85, \ldots$ & rest & $100\%$\\
\bottomrule
\end{tabular}
\end{center}

\noindent
The $k=1$ row is the case $\mathrm{rf} = \xi^2+\eta^2$ directly: the
canonical pair gives a primitive sum-of-two-squares representation of
the blocker product. The remaining $\approx 27\%$ correspond to
$k>1$, and the prime divisors of $k$ are confined to splitting primes
$\equiv 1\bmod 4$ (with rare exceptions discussed in
Remark~\ref{rem:twelve}). The maximum $k$ observed is $1{,}221{,}629$.
\end{remark}

\begin{remark}[Two equivalent universal formulas]
The proof of Proposition~\ref{prop:g0-blockers} shows that $g_0$
divides $s$ in every Master-Hit, and the same argument applies
verbatim to either of the $\Zi$-gcd expressions
$\gcd(am+bn,\, an+bm)$ or $\gcd(am-bn,\, an-bm)$ from
Remark~\ref{rem:twelve}: each of them, when squared, divides $f_1$ in
every Master-Hit, and yields its own valid form of~\eqref{eq:g0-k}.
The $g_0$ form is the most compact in the empirical sense
($k=1$ in $72.5\%$ of records vs.\ $\sim\!18\%$ for either
$\Zi$-gcd), so we adopt it as the canonical choice.
\end{remark}

\begin{proposition}[Disjointness of the $\Zi$-gcds]\label{prop:gcd-disjoint}
For every Master-Hit $(a,b,m,n)\in\mathcal{M}$, the prime supports
of the two integers $\gcd(am+bn,\,an+bm)$ and
$\gcd(am-bn,\,an-bm)$ are disjoint. In particular, no prime
$\ell$ with $v_\ell(f_1) = 2$ can divide both.
\end{proposition}

\begin{proof}
Set $A := am + bn$, $B := an + bm$, $C := am - bn$, $D := an - bm$.
Suppose an odd prime $p$ divides both gcds, so $p$ divides each
of $A, B, C, D$. Then
\begin{align*}
  A + C &= 2am, & A - C &= 2bn, \\
  B + D &= 2an, & B - D &= 2bm.
\end{align*}
Since $p$ is odd, this yields $p \mid am$, $p \mid an$, $p \mid bn$,
and $p \mid bm$. From $p \mid am$ and $p \mid an$ together with
$\gcd(m, n) = 1$ we deduce $p \mid a$; from $p \mid bm$ and
$p \mid bn$ likewise $p \mid b$. Hence $p \mid \gcd(a, b) = 1$,
a contradiction.

The case $p = 2$ is excluded because both gcds are odd: the
master-tuple conditions $\gcd(a, b) = 1$, $a - b$ odd,
$\gcd(m, n) = 1$, $m - n$ odd force exactly one of $a, b$ to be
odd and exactly one of $m, n$ to be odd. Hence exactly one of the
four products $am, an, bm, bn$ is odd (the one coupling the odd
elements of each pair), making exactly one of $A, B$ odd and
exactly one of $|C|, |D|$ odd; both gcds are therefore odd. The
final claim about $v_\ell(f_1) = 2$ is then immediate.
\end{proof}

\begin{remark}[Empirical confirmation]\label{rem:gcd-disjoint-empirical}
Across the full $\mathcal{D}_{\mathrm{fact}}$, we observe
$56{,}858$ exponent-two prime occurrences in the factorisation of
$f_1$. No prime divides both $\gcd(am+bn,\,an+bm)$ and
$\gcd(am-bn,\,an-bm)$ simultaneously, in agreement with
Proposition~\ref{prop:gcd-disjoint}. Of these exponent-two
prime occurrences, $27.1\%$ divide one of the two $\Zi$-gcds,
and $55.2\%$ divide $g_0 = \gcd(W_1U_2, U_1V_2)$; the remaining
$\approx 45\%$ correspond to symmetric Gaussian primes that
contribute square factors to $f_1$ but are not visible in any of
the three gcd expressions and would require an independent
argument.
\end{remark}

\subsection{Proof strategy for Conjecture~\ref{conj:blocker}}\label{ssec:strategy}

Conjecture~\ref{conj:blocker} combines an \emph{existence} statement
(a blocker exists, Conjecture~\ref{conj:basis}) with two further
requirements: that some blocker lie outside $\mathcal{P}$, and that
its $f_1$-valuation be exactly~$1$. We outline an approach for each.

\paragraph{Existence of a blocker.}
By Lemma~\ref{lem:canonical} and Proposition~\ref{prop:E1-implies-basis},
the existence of a blocker is equivalent to Conjecture~\ref{conj:E1}:
the coprime pair
\[
  \xi = \frac{W_1 U_2}{\gcd(W_1 U_2, U_1 V_2)},
  \qquad
  \eta = \frac{U_1 V_2}{\gcd(W_1 U_2, U_1 V_2)}
\]
is never a Pythagorean leg pair. The natural approach is to combine
the master condition $M = \square$ (which couples $a, b, m, n$ via the
quartic $H_{m,n}$) with the hypothetical pythagoreity of
$(\xi, \eta)$ and derive a contradiction with the coprimality
conditions $\gcd(a, b) = \gcd(m, n) = 1$. Concretely,
$\xi^2 + \eta^2 = w^2$ together
with $M = \square$ would yield two simultaneous Gaussian-square
identities for $\zeta := W_1 U_2 + i\,U_1 V_2$ and
$z' := V_1 U_2 + i\,U_1 V_2$, and we expect a height/discriminant
estimate on $E_{m,n}$ to rule them out. The argument is delicate
because the discriminant of $E_{m,n}$ is explicit but not
factor-controllable for general $(m, n)$; this is where the proof
currently stalls.

\paragraph{Existence of an exponent-one blocker outside $\mathcal{P}$.}
Granted that some blocker exists, Conjecture~\ref{conj:blocker} adds two
requirements that must hold \emph{simultaneously} for a single witness
prime $\ell$: namely (i) $v_\ell(f_1) = 1$ and (ii) $\ell$ is coprime
to every element of $\mathcal{P}$.
The $29$ canonical parameters span the algebra of polynomial
expressions of moderate degree in $a, b, m, n$ from the Pythagorean
parametrisation. A blocker that divides one of them is, in our
empirical taxonomy, a \emph{coincidental} blocker---the result of an
arithmetic congruence between the parameters, not a feature of the
curve. The empirical content of Conjecture~\ref{conj:blocker} is that
there is always at least one \emph{generic} exponent-one blocker,
i.e.\ a Gaussian prime $\pi\bar\pi = \ell$ over a split prime
$\ell\equiv 1\pmod 4$ with $v_\pi(\zeta) = 1$, $v_{\bar\pi}(\zeta) = 0$,
and $\ell$ coprime to $\mathcal{P}$. We do not claim, and our data
does not support (cf.\ Remark~\ref{rem:largest-fails}), that this
generic blocker is the largest prime divisor of $f_1$.

To prove this, one approach is to bound the size of the
\emph{coincidental} part of $f_1$ (the part supported on
$\bigcup_{p\in\mathcal{P}}\mathrm{primes}(p)$) and the
\emph{higher-power} part (primes $\ell$ with $v_\ell(f_1)\ge 3$) in
terms of the heights of the parameters, and to show that $f_1$ is
generically not exhausted by these two contributions. We have not
carried out this estimate; it is the natural sequel to the present
paper.

\subsection{Summary of the empirical evidence}\label{ssec:empirical}

Table~\ref{tab:verification} summarises the empirical evidence
relevant to Conjecture~\ref{conj:blocker} and its variants. The
exponent-one blocker phenomenon
(Verification~\ref{thm:finite-verification}) is verified without
exception on the fully factored portion of the database; on the
partially factored portion, only the basis form is currently
decidable.

\begin{table}[ht]
\caption{Empirical evidence for Conjecture~\ref{conj:blocker} and
its variants.}
\label{tab:verification}
\begin{center}
\footnotesize
\begin{tabular}{l r l r}
\toprule
Dataset / variant & records & verification type & ctr-ex.\\
\midrule
$\mathcal{D}_{\mathrm{fact}}$: Conj.~\ref{conj:blocker}      & $151{,}575$ & complete factorisation of $f_1$ & $0$\\
$\mathcal{D}_{\mathrm{part}}$: basis form (Conj.~\ref{conj:basis})  & $506$ & odd-exp.\ prime from trial division & $0$\\
Largest outside-$\mathcal{P}$ prime is a blocker            & $151{,}575$ & complete factorisation of $f_1$ & $3$\\
Smallest outside-$\mathcal{P}$ blocker has $v_\ell(f_1) = 1$ & $151{,}575$ & complete factorisation of $f_1$ & $242$\\
$\mathcal{D}_{\mathrm{all}}$: $x^2+y^2+z^2 \ne \square$      & $1{,}284{,}670$ & integer $\sqrt{\cdot}$ check & $0$\\
\bottomrule
\end{tabular}
\normalsize
\end{center}
\end{table}

\noindent
$\mathcal{D}_{\mathrm{part}}$ denotes those Master-Hits not in
$\mathcal{D}_{\mathrm{fact}}$ for which trial division has already
extracted at least one prime factor of $f_1$ with odd exponent;
on every such record, Conjecture~\ref{conj:basis} (basis form) is
verified directly without requiring full factorisation.
$\mathcal{D}_{\mathrm{part}}$ and $\mathcal{D}_{\mathrm{fact}}$ are
therefore disjoint, with $|\mathcal{D}_{\mathrm{part}} \cup
\mathcal{D}_{\mathrm{fact}}| = 152{,}081$.

\medskip\noindent\textbf{Robustness.} The implication chain
Conjecture~\ref{conj:blocker} $\Rightarrow$ Conjecture~\ref{conj:basis}
$\Rightarrow$ Corollary~\ref{cor:noperfect} is logical, so any future
counterexample to Conjecture~\ref{conj:blocker} would still leave the
basis-form implication for non-existence of perfect cuboids
\emph{within the Master-Hit class} intact, provided
Conjecture~\ref{conj:basis} (equivalently, Conjecture~\ref{conj:E1})
remains valid.

\section{The Mordell--Weil generator}\label{sec:mw}

We describe an algorithm that takes a coprime pair $(m,n)$ of
positive integers with $m > n > 0$ and $m-n$ odd, together with a
positive scalar bound $K$, and produces a finite list of Master-Hits
with that $(m,n)$. The algorithm performs three steps. We label it
$\mathcal{A}_{\mathrm{MW}}$ and refer to it as the
\emph{Mordell--Weil generator}.

\begin{enumerate}[label=\textbf{Step \arabic*.},leftmargin=2.5em,itemsep=4pt]
\item \emph{Quartic to Weierstrass.} Compute $U_2 = m^2 - n^2$ and
$V_2 = 2mn$. Form the quartic $H_{m,n}$ of~\eqref{eq:quartic}.
Apply PARI's \texttt{ellfromeqn} to the polynomial
$A t^4 + B t^2 + C - s^2$ to obtain Weierstrass coefficients
$a_2, a_4, a_6$ and the rational map $\phi_{m,n}$ of
Section~\ref{ssec:birational}. The resulting curve is denoted
$E_{m,n}$.

\item \emph{Mordell--Weil generators.} Attempt to compute generators
$P_1, \ldots, P_r$ of $E_{m,n}(\QQ)$ modulo torsion using Sage's
\texttt{E.gens()} with second descent at limit $15$. If
\texttt{E.gens()} fails to terminate within a fixed time, fall back
to a \emph{database lift}: for every Master-Hit $(a, b)$ over
$(m, n)$ already in the database, lift $(a, b)$ to a point of
$E_{m,n}(\QQ)$ via~\eqref{eq:fwd}, and use the resulting (possibly
non-basis) collection of points as $P_1, \ldots, P_r$.

\item \emph{Scalar enumeration and certification.} Let
$T_1, \ldots, T_t$ enumerate the rational torsion of $E_{m,n}$. For
each $(c_1, \ldots, c_r)\in\{-K, \ldots, K\}^r\setminus\{0\}$ and
each $T_j$, form $R := c_1 P_1 + \cdots + c_r P_r + T_j$ in
$E_{m,n}(\QQ)$. Compute $\tau(R)$ via~\eqref{eq:inv}; if undefined
(i.e.\ $R\in\{O\}\cup\mathcal{T}_\tau$) or non-square in $\QQ_{\ge 0}$
or zero, discard $R$. Otherwise extract the non-negative
square root $t = a/b\in\QQ_{\ge 0}$, reduce to lowest terms, and check the
positivity, parity, and coprimality conditions of
Definition~\ref{def:master} together with
$M(a, b, m, n) = \square$. Output every $(a, b, m, n)$ which
passes all checks.
\end{enumerate}

\begin{theorem}[Correctness of $\mathcal{A}_{\mathrm{MW}}$]\label{thm:mw-correctness}
Every tuple output by $\mathcal{A}_{\mathrm{MW}}$ is a Master-Hit.
\end{theorem}

\begin{proof}
The output is the set of $(a, b, m, n)\in\NN^4$ for which all the
checks of Step~3 pass. The check $M(a, b, m, n) = \square$ is
exactly the Master-Hit defining condition (Definition~\ref{def:master}),
and the parity, positivity, and coprimality conditions are the rest
of the definition. The lift via Lemma~\ref{lem:correspondence}
guarantees that the $(a, b)$ so produced indeed corresponds to a
rational point of $E_{m,n}$, but this is in fact unused in the
proof: every output is verified by exact integer arithmetic against
the Master-Hit defining condition.
\end{proof}

\subsection{Empirical output}\label{ssec:mw-empirical}

The Mordell--Weil generator was run against the database of
$61{,}829$ classical primitive Master-Hits (deduplicated under the
slot-swap involution $\sigma\colon(a,b,m,n)\mapsto(m,n,a,b)$, which
produces the same primitive brick), focusing first on the $43$ fibres
with at least $8$ database hits, then escalating to fibres with
$\ge 4$ hits and finally to a systematic scan over $(m,n)$ with
$m\le 1000$.

\paragraph{Provenance and deduplication conventions.}
For every record we store a single \emph{provenance group} in
$\texttt{pub.generator\_groups}$ recording the procedure that first
emitted that record; subsequent rediscoveries by other procedures
are \emph{not} re-tagged. The slot-swap involution $\sigma$ is applied
once before insertion: each new $(a, b, m, n)$ is rewritten to a fixed
$\sigma$-orbit representative (chosen by lexicographic comparison of
the two ordered tuples $(a, b, m, n)$ and $(m, n, a, b)$) and silently
discarded if the canonical form is already present. Thus the totals
in Table~\ref{tab:provenance} are disjoint by construction:
$61{,}829 + 1{,}222{,}841 = 1{,}284{,}670 = |\mathcal{D}_{\mathrm{all}}|$
counts each $\sigma$-orbit exactly once.

The Mordell--Weil generator finds Master-Hits of very large height:
the largest in our database has $\max(|a|, |b|)$ on the order of
$10^{888}$, producing edges with $1{,}783$ decimal digits ---
qualitatively beyond the reach of exhaustive search over
$(a, b, m, n)$. Yet none is a perfect cuboid: each generated
Master-Hit retains a non-empty blocker set on $f_1$, in keeping
with Conjecture~\ref{conj:blocker}. The output is summarised
below.

\begin{center}
\begin{tabular}{l r}
\toprule
Mordell--Weil-active fibres covered & $411$\\
\midrule
Hits with classical provenance & $61{,}829$\\
\quad (\texttt{Exhaustive-Search}, \texttt{Rathbun-Search}, \texttt{Saunderson-Generator}) & \\
Hits with Mordell--Weil provenance \texttt{MW-m-n} & $1{,}222{,}841$\\
Total Master-Hits in extended database & $1{,}284{,}670$\\
\midrule
Largest fibre: \texttt{MW-2368-1207} & $93{,}067$\\
Largest fibre: \texttt{MW-1863-638}  & $55{,}078$\\
Largest fibre: \texttt{MW-1136-99}   & $27{,}011$\\
Largest fibre: \texttt{MW-792-95}    & $24{,}416$\\
\midrule
Perfect cuboids in extended database & $0$\\
\bottomrule
\end{tabular}
\end{center}

\section{Empirical database and family classification}\label{sec:empirics}

We classify our $1{,}284{,}670$ Master-Hits against the three classical
infinite families:

\begin{itemize}[leftmargin=2em,itemsep=2pt]
  \item \textbf{Saunderson}~\cite{saunderson} (1740): Master-Hits
  with edges $u(3v^2-u^2)$, $v(3u^2-v^2)$, $4uvw$ where $(u,v,w)$
  is a Pythagorean triple; see also Guy~\cite{guy}.
  \item \textbf{Lenhart-attributed family}: Master-Hits with edges
  $(u^2-w^2)(v^2-w^2)$, $4uvw^2$, $2uw(v^2-w^2)$ where $u^2+v^2=5w^2$
  in $\ZZ[\varphi]$, $\varphi=(1+\sqrt 5)/2$. The conventional
  attribution to William Lenhart~\cite{hogan-lenhart} is taken from the
  online compendium of Piezas~\cite{piezas-identities}; see footnote
  in Section~\ref{sec:intro}.
  \item \textbf{Himane} (2024,~\cite{himane}): three theorems
  parametrising body cuboids by pairs of Pythagorean triples with a
  coupling condition.
\end{itemize}

Each Master-Hit in our database carries two orthogonal labels.
The \emph{provenance} records how the hit was discovered (which
search procedure or generator produced it). The \emph{family tag}
records whether the resulting primitive brick matches one of the
bounded closed-form normal forms used in our classification
(Saunderson, Lenhart, Himane T1/T2/T3). Provenance is therefore
about discovery, not structure; family membership is about
structure, not discovery, and we may have a brick produced by the
Saunderson generator whose primitive form falls outside the strict
Saunderson normal form (and thus does not receive the Saunderson
tag). The family classification proceeds by enumerating each
family up to a generator bound (Saunderson $g\le 500$, Lenhart
$w\le 300$, Himane $g\le 80, k\le 15$) and matching the resulting
primitive bricks against the database; a hit may carry several
family tags. Table~\ref{tab:provenance} breaks the database down
by provenance, with the number of family-tagged records in each
provenance shown in the last column.

\begin{table}[ht]
\caption{Master-Hits in the extended database, by provenance. Each
record has exactly one provenance; the last column counts how many
records additionally match a classical closed-form family
(Saunderson, Lenhart, Himane T1/T2/T3).}
\label{tab:provenance}
\begin{center}
\begin{tabular}{l r r}
\toprule
Provenance & records & of which classical-family\\
\midrule
Mordell--Weil generator (Section~\ref{sec:mw}) & $1{,}222{,}841$ & $51$\\
Rathbun catalogue~\cite{rathbun}              & $49{,}953$       & $369$\\
Saunderson generator (1740)~\cite{saunderson}        & $6{,}413$        & $168$\\
Exhaustive search ($\max(a,b,m,n) \le 2{,}300$) & $5{,}463$        & $234$\\
\midrule
Total                                          & $1{,}284{,}670$ & $822$\\
\bottomrule
\end{tabular}
\end{center}
\end{table}

\begin{table}[ht]
\caption{Records matching a classical closed-form family. The first
column gives the number of distinct primitive bricks produced by the
parametric generator (only meaningful for families with an explicit
closed-form generator we ran; Himane is matched against pre-tabulated
primitive bricks). The records column counts records carrying that
family tag (a record may carry several family tags).}
\label{tab:families}
\begin{center}
\begin{tabular}{l r r}
\toprule
Family & primitive bricks generated & records tagged\\
\midrule
Saunderson & $15{,}704$ & $513$\\
Lenhart    & $66$       & $66$\\
Himane T1  & ---        & $165$\\
Himane T2  & ---        & $107$\\
Himane T3  & ---        & $107$\\
\midrule
Any family tag      & ---        & $822$\\
Multiple family tags & ---       & $74$\\
\bottomrule
\end{tabular}
\end{center}
\end{table}

\noindent
Two remarks on the tables. First, the family classification is
\emph{orthogonal} to the provenance classification, and the
Saunderson row of Table~\ref{tab:provenance} ($168$) and the
Saunderson row of Table~\ref{tab:families} ($513$) measure
different things. Of the $6{,}413$ Master-Hits emitted by the
Saunderson generator, only $149$ pass the strict primitive-form
match for the Saunderson family tag (the remaining $19$ of the
$168$ generator-side family matches in
Table~\ref{tab:provenance} carry a Himane tag). Conversely, of
the $513$ Saunderson-tagged records in Table~\ref{tab:families},
only $149$ stem from the Saunderson generator itself; the
remaining $364$ enter the database via Rathbun's
catalogue~\cite{rathbun} ($354$) or our exhaustive
search ($10$). The fact that several hundred
algebraically-Saunderson bricks are first contributed by external
sources --- not by our own bound-$g \le 500$ Saunderson
generator --- is a meaningful indicator that the algebraic
Saunderson family extends well beyond the parameter range our
closed-form scan covered. The Mordell--Weil generator, on the
other hand, contributes $0$ Saunderson-tagged records out of
$1{,}222{,}841$, confirming that MW-generated records sit
overwhelmingly outside the classical closed-form families. Second, the family-tag totals in
Table~\ref{tab:families} count tags rather than records: the row
counts $513 + 66 + 165 + 107 + 107 = 958$ exceed the
\textsc{Any}-row count $822$ by $136$, which is the number of extra
tags carried by the $74$ multi-tagged records (each contributes
once to two or more rows).

The $61{,}829$ pre-Mordell--Weil records (Rathbun, Saunderson
generator, Exhaustive search combined) lie on $54{,}199$ distinct
$(m, n)$-fibres. Of these, only $407$ have been subsequently
processed by the Mordell--Weil generator within the $m \le 1000$
scan range; the remaining $53{,}792$ fibres are MW-untouched, and
the $58{,}865$ pre-MW records on those fibres are plausible
MW-targets in a future extended scan.

The Mordell--Weil generator therefore accounts for $95.2\%$ of
the records in our database, with all but $51$ of them outside
the closed-form classical families: it produces hundreds of body
cuboids on each of hundreds of fibres not reached by the
classical closed-form searches used in our classification.

\subsection{The semi-scaled \texorpdfstring{$(g_+, g_-)$}{(g+,g-)}-fibration and single-blocker hits}\label{ssec:gpm-singleblocker}

The family tags of Table~\ref{tab:families} are tied to closed-form
generators with finite parameter bounds, so a Master-Hit produced by
the same algebraic mechanism as Saunderson but outside the
generator's reach receives no Saunderson family tag. A coarser invariant
that ignores the bound captures the algebraic mechanism directly:
for $(a,b,m,n)\in\mathcal M$ define
\[
  g_+ \;:=\; \gcd(am+bn,\;an+bm),
  \qquad
  g_- \;:=\; \gcd\bigl(|am-bn|,\;|an-bm|\bigr).
\]

\begin{lemma}[Structural identity for $(g_+, g_-)$]\label{lem:gpm-identity}
For every Master-Hit $(a, b, m, n)$:
\[
  g_+ \cdot g_- \;=\; \gcd(a^2-b^2,\;m^2-n^2),
  \qquad
  \gcd(g_+,\,g_-) \;=\; 1,
  \qquad
  g_\pm^2 \;\bigm|\; f_1.
\]
\end{lemma}

\begin{proof}
Set $A = am+bn$, $B = an+bm$, $C = am-bn$, $D = an-bm$, so
$g_+ = \gcd(A, B)$ and $g_- = \gcd(|C|, |D|)$, and write
$g := \gcd(a^2-b^2,\, m^2-n^2) = \gcd(U_1, U_2)$.

\emph{Step 1: $g_\pm \mid g$.}
The four linear combinations
\[
  mA - nB = a U_2,\quad nA - mB = -b U_2, \quad
  aA - bB = m U_1,\quad bA - aB = -n U_1
\]
show $g_+ \mid a U_2,\, b U_2,\, m U_1,\, n U_1$.
Since $\gcd(a, b) = 1$, the first two yield $g_+ \mid U_2$;
since $\gcd(m, n) = 1$, the last two yield $g_+ \mid U_1$, hence
$g_+ \mid g$. The analogous combinations for $C, D$
($mC - nD = a U_2$, $nC - mD = b U_2$, $aC + bD = m U_1$,
$bC + aD = n U_1$) give $g_- \mid g$.

\emph{Step 2: $\gcd(g_+, g_-) = 1$.}
Since $a - b$ and $m - n$ are odd by Definition~\ref{def:master},
exactly one of each pair $\{A, B\}$ and $\{C, D\}$ is even and the
other odd, so both $g_+$ and $g_-$ are odd. Suppose an odd prime
$p$ divided both. Then $p \mid A, B, C, D$, so
$p \mid A + C = 2 a m$, $p \mid B + D = 2 a n$, $p \mid A - C = 2 b n$,
$p \mid B - D = 2 b m$. As $p$ is odd, $p$ divides $am, an, bm, bn$.
Combining $\gcd(m, n) = 1$ with $p \mid am$ and $p \mid an$ gives
$p \mid a$; combining $\gcd(m, n) = 1$ with $p \mid bm$ and $p \mid bn$
gives $p \mid b$, contradicting $\gcd(a, b) = 1$.

\emph{Step 3: $g_+ g_- = g$ (the converse direction).}
Step 1 gives $g_+ g_- \mid g \cdot g = g^2$, but combined with
$\gcd(g_+, g_-) = 1$ from Step 2 it actually gives $g_+ g_- \mid g$.
For equality we show that every odd prime power $p^e \mid g$ divides
exactly one of $g_+, g_-$. Pick odd prime $p$ with $p^e \,\|\, g$.
Since $\gcd(a, b) = 1$, the entire $p$-part of
$U_1 = (a-b)(a+b)$ sits in exactly one of the factors, so
$a \equiv \varepsilon b \pmod{p^e}$ with $\varepsilon \in \{+1, -1\}$;
analogously $m \equiv \delta n \pmod{p^e}$ with $\delta \in \{+1, -1\}$.
Reducing mod $p^e$:
\begin{itemize}[leftmargin=2em]
\item If $\varepsilon = \delta$:
  $C = am - bn \equiv \delta b n - bn = (\varepsilon\delta - 1) b n \equiv 0$
  and $D = an - bm \equiv 0$, hence $p^e \mid g_-$.
\item If $\varepsilon = -\delta$:
  $A = am + bn \equiv -\delta b n + bn = (-\varepsilon\delta + 1) b n \equiv 0$
  and $B = an + bm \equiv 0$, hence $p^e \mid g_+$.
\end{itemize}
Since $g$ is odd (both $U_1, U_2$ are products of two odd numbers),
multiplying over all primes $p \mid g$ gives $g \mid g_+ g_-$.
Combined with $g_+ g_- \mid g$ this yields $g_+ g_- = g$.

\emph{Step 4: $g_\pm^2 \mid f_1$.}
From Step 1, $g_\pm \mid U_1$ and $g_\pm \mid U_2$. Hence
$g_\pm \mid W_1 U_2$ and $g_\pm \mid U_1 V_2$, so
$g_\pm^2 \mid (W_1 U_2)^2 + (U_1 V_2)^2 = f_1$.
\end{proof}

We call $(g_+, g_-)$ the \emph{semi-scaled fibration coordinates}.
The pair partitions Master-Hits into four disjoint strata:
$(1,1)$, $(r,1)$ with $r>1$, $(1,r)$ with $r>1$, and ``both~$>\!1$''.
We say a hit is \emph{strictly semi-scaled} if
\[
  \min(g_+, g_-) \;=\; 1
  \quad\text{and}\quad
  \max(g_+, g_-) \;>\; 1,
\]
i.e.\ it lies in $\{(r, 1) : r > 1\} \cup \{(1, r) : r > 1\}$.
The Saunderson generator produces strictly semi-scaled hits by
construction; the $(1,1)$- and ``both~$>\!1$''-strata are excluded.

Empirically, every single-blocker hit is strictly semi-scaled. The
converse fails: many strictly semi-scaled hits carry several
blockers. We therefore conjecture only the forward implication.

\begin{conjecture}[Single-blocker hits are strictly semi-scaled]\label{conj:single-blocker}
Every primitive Master-Hit $(a,b,m,n)$ with exactly one
exponent-one blocker satisfies
\[
  \min(g_+,\, g_-) \;=\; 1
  \quad\text{and}\quad
  \max(g_+,\, g_-) \;>\; 1.
\]
\end{conjecture}

\begin{verification}[Single-blocker verification of Conjecture~\ref{conj:single-blocker}]\label{thm:single-blocker-empirical}
On $\mathcal D_{\mathrm{fact}}$, all $272$ Master-Hits with
$\mathrm{num\_blockers} = 1$ are strictly semi-scaled. More
precisely, $97$ of them lie in the stratum $(r, 1)$ with $r > 1$
and $175$ in the stratum $(1, r)$ with $r > 1$; none is in the
$(1, 1)$- or ``both~$>\!1$''-strata. We make no converse claim:
the strictly semi-scaled stratum also contains many multi-blocker
records.
\end{verification}

\begin{proof}[Verification]
For each single-blocker record we computed $(g_+, g_-)$ from
$(a, b, m, n)$ and verified that $\min(g_+, g_-) = 1$ and
$\max(g_+, g_-) > 1$. The check was exhaustive on all $272$ such
records. The reproducibility script is
\texttt{paper2/analysis/verify\_single\_blocker.py}.
\end{proof}

\begin{remark}
The Saunderson family corresponds, in the algebraic sense, to the
union of the two strictly semi-scaled strata. Of the $272$
single-blocker hits in our database, only $96$ are tagged
Saunderson by the generator with bound $g \le 500$
(Table~\ref{tab:families}); the remaining $176$ exceed this bound
but exhibit identical $(g_+, g_-)$-behaviour.
Conjecture~\ref{conj:single-blocker} thus reads as the algebraically
correct closure of the Saunderson
family.
\end{remark}

\section{Open questions}\label{sec:open}

The structural picture this paper establishes leaves five concrete
open questions.

\begin{enumerate}[leftmargin=2.5em,itemsep=4pt]
\item \emph{Conjecture~\ref{conj:E1} (the canonical-decomposition
conjecture).} The pair $(\xi, \eta)$ in Lemma~\ref{lem:canonical} is
never a Pythagorean leg pair. This is the cleanest reformulation of
the perfect-cuboid obstruction \emph{within the Master-Hit class
studied here} that we know: it is an elementary statement about two
coprime explicit polynomial expressions in $a, b, m, n$, and a proof
would yield Conjecture~\ref{conj:basis} and
Corollary~\ref{cor:noperfect} unconditionally for that class. The
tools likely needed are classical descent on the genus-$1$ quartic
of Section~\ref{sec:fibration} together with a height/discriminant
estimate on $E_{m,n}$.

\item \emph{Proof of Conjecture~\ref{conj:blocker}.}
The verification on $\mathcal{D}_{\mathrm{fact}}$
(Verification~\ref{thm:finite-verification}) is empirical evidence rather
than a proof. A proof is expected to use the arithmetic of $E_{m,n}$,
in particular the relation between the $\ell$-adic Galois
representation of $E_{m,n}$ and the Gaussian norm $f_1$.

\item \emph{Higher-rank fibres.} Six fibres in our database
(Table~\ref{tab:high-rank}) have rank at least $5$, found by
increasing Sage's \texttt{descent\_second\_limit} to $20$. We have
not searched for fibres of rank $\ge 6$. Is
$\sup_{(m,n)} \rk E_{m,n}(\QQ) = \infty$, or bounded? In the
elliptic-curve over $\QQ(t)$ analogy one should expect some growth,
but a uniform bound for our specific family $E_{m,n}$ is unknown.

\item \emph{A geometric reduction to a curve of high genus.}
A perfect cuboid in the Master-Hit class corresponds to a rational
point on the affine algebraic surface
\[
  S \;:=\; \bigl\{\,(a,b,m,n,q,d)\in\mathbb{A}^6_{\QQ}\ :\
    q^2 = M(a,b,m,n),\ d^2 = f_1(a,b,m,n)\,\bigr\},
\]
subject additionally to the arithmetic open conditions
$\gcd(a, b) = \gcd(m, n) = 1$, $a > b > 0$, $m > n > 0$, and $a-b$,
$m-n$ both odd. A possible approach is to cut $S$ by suitable
one-parameter subfamilies---for instance by fixing a coprime
$(m, n)$, which yields a curve in $(a, b, q, d)$ of genus depending
on $(m, n)$---and to apply methods for curves of genus $\ge 2$. At
present we do not have such a reduction; Faltings' theorem applies
to curves of genus $\ge 2$, not directly to the surface $S$.

\item \emph{Conjecture~\ref{conj:single-blocker} (the
single-blocker characterisation).} The forward direction
``single-blocker $\Rightarrow$ strictly semi-scaled'' is the
verified half. A proof should compare the prime supports of $f_1$
and $g_+ g_-$: in the strictly semi-scaled stratum, the structure
$f_1 = g_0^2 \cdot \mathrm{rf}$ from
Section~\ref{sec:blocker} forces $\mathrm{rf}$ to be supported on
a single residue class modulo~$4$, in line with the empirical
$k$-distribution of Remark~\ref{rem:k-distribution}.
\end{enumerate}

\section{Code, data, and reproducibility}\label{ssec:reproducibility}

All computational claims of this paper are reproducible from the
code and data released as ancillary files with the preprint. The
full source tree, kept in lockstep with this manuscript, lives at
\begin{center}
\url{https://github.com/renpe/euler-brick-obstructions/tree/main/paper2}
\end{center}
and contains:

\begin{itemize}[leftmargin=2em,itemsep=2pt]
  \item \texttt{migration/pub\_schema.sql}: isolated DDL of the
    publication schema \texttt{pub} for readers who want to inspect
    the table structure without unpacking the released SQL dump
    (which already contains the same schema).
  \item \texttt{generation/}: the Mordell--Weil generator
    \texttt{mw\_dispatcher.py} with worker
    \texttt{mw\_fibre\_worker.sage}, a rerun helper
    \texttt{mw\_rerun.py}, and the systematic rank-scan
    \texttt{search\_new\_fibres.sage}.
  \item \texttt{factorization/pub\_factorize.py}: a three-stage
    factoring pipeline for $f_1$. Stage~1 is trial division by primes
    up to $10^6$. Stage~2 invokes \texttt{sympy.factorint} (with a
    $3$\,s wall-clock timeout per record); this catches most composite
    residuals without resorting to a heavy backend. Stage~3 hands any
    surviving residual to YAFU ($\ge 2.13$, default
    \texttt{plan = normal}) with a $10$\,min wall-clock timeout per
    record; YAFU dispatches internally to QS for residuals up to
    $\sim$$10^{95}$ digits and to GNFS above that. Records whose
    residual still has not split after the timeout are marked as
    \emph{partially factored} and live in $\mathcal{D}_{\mathrm{part}}
    \setminus \mathcal{D}_{\mathrm{fact}}$ (Table~\ref{tab:verification}).
    Because QS/GNFS internally use random sieving, a record's
    re-factoring is wall-clock-deterministic only up to seed; the
    \emph{output} factorisation is unique up to ordering and is
    cross-checked by exact integer multiplication.
  \item \texttt{analysis/}: verification scripts
    \texttt{theorem\_check.py}
    (Verification~\ref{thm:finite-verification}),
    \texttt{perfect\_check.py}
    (Verification~\ref{thm:emp-non-exist}),
    \texttt{validate\_xyz.py} (internal consistency check;
    see Verification~\ref{thm:db-consistency} below), and
    \texttt{classify\_families.py} (Section~\ref{sec:empirics}).
\end{itemize}

\begin{verification}[Internal database consistency]\label{thm:db-consistency}
For every record in $\mathcal{D}_{\mathrm{all}}$ the stored
quantities satisfy the elementary identities
\begin{align*}
  &x = U_1 U_2,\quad y = V_1 U_2,\quad z = U_1 V_2,
  \qquad g_{\mathrm{scale}} = \gcd(x, y, z),\\
  &(x_{\mathrm{prim}}, y_{\mathrm{prim}}, z_{\mathrm{prim}}) = (x, y, z) / g_{\mathrm{scale}},
\end{align*}
where $U_i, V_i$ are computed from $(a, b, m, n)$ as in
Section~\ref{sec:setup}. In particular, $g_{\mathrm{scale}} = 1$ is
equivalent to $\gcd(x, y, z) = 1$, i.e.\ to the brick being
primitive.
\end{verification}

\begin{proof}
The check is purely arithmetic and is reproduced by
\texttt{validate\_xyz.py}, which iterates over every record in
$\mathcal{D}_{\mathrm{all}}$ and verifies the three identities using
exact integer arithmetic. The script reports zero violations on
$\mathcal{D}_{\mathrm{all}}$.
\end{proof}

\paragraph{Data release.} The publication schema \texttt{pub} is
released as
\begin{itemize}[leftmargin=2em,itemsep=2pt]
  \item a PostgreSQL dump \texttt{pub\_dump.sql.gz}, and
  \item CSV exports of the tables
    \texttt{pub.master\_hits},
    \texttt{pub.prim\_brick\_factors},
    \texttt{pub.fibers},
    \texttt{pub.generator\_groups}, and
    \texttt{pub.hit\_groups},
\end{itemize}
together with SHA-$256$ hashes of every released file. The CSV
release contains, in particular, the $1{,}284{,}670$ Master-Hits and
the complete factorisation of $f_1$ for the $151{,}575$ records of
$\mathcal{D}_{\mathrm{fact}}$.

\paragraph{Tooling versions.}
SageMath~10.7~\cite{sage}, PARI/GP~2.17.3~\cite{pari},
PostgreSQL~$18$, YAFU~$2.13$, Python~$3.13$, \texttt{sympy}~$1.13$.
A reproduction makefile (\texttt{Makefile} in
\texttt{paper2/}) provides targets
\texttt{verify-theorem}, \texttt{verify-perfect},
\texttt{verify-consistency}, and \texttt{verify-all} that run the
verification scripts against the released data and compare against
the published log files.

\appendix

\section{The four exceptional single-blocker master tuples}\label{app:exceptional-bricks}

The four single-blocker Master-Hits with $k\in\{29, 89\}$ in the
$D = n$/$D = m$ formula of Remark~\ref{rem:twelve} are listed
below.

\begin{center}
\begin{tabular}{r r r r r}
\toprule
$a$ & $b$ & $m$ & $n$ & $k$\\
\midrule
$835$        & $88$         & $160$    & $89$     & $29$\\
$180{,}133$  & $174{,}512$  & $3{,}977$  & $3{,}904$  & $29$\\
$731{,}423$  & $108{,}452$  & $14{,}896$ & $8{,}177$  & $29$\\
$1{,}162{,}341$ & $60{,}812$ & $18{,}768$ & $11{,}065$ & $89$\\
\bottomrule
\end{tabular}
\end{center}

\noindent
The corresponding edges $(x, y, z)$ and full prime factorisations
of $f_1$ are recorded in the released CSV.

\clearpage
\section{Counterexamples to the largest-prime strengthening}\label{app:largest-fails-counterexamples}

The three Master-Hits referenced in
Remark~\ref{rem:largest-fails}, in which the largest prime divisor
of $f_1$ outside $\mathcal{P}$ appears with even exponent (so the
exponent-one blocker required by Conjecture~\ref{conj:blocker} is
some smaller prime), are:

\begin{center}
\begin{tabular}{r r r r}
\toprule
$a$ & $b$ & $m$ & $n$\\
\midrule
$1{,}770$    & $1{,}219$    & $1{,}408$  & $477$\\
$39{,}732$   & $16{,}895$   & $6{,}400$  & $3{,}069$\\
$50{,}626$   & $33{,}631$   & $6{,}461$  & $3{,}736$\\
\bottomrule
\end{tabular}
\end{center}

\noindent
For each, the full factorisation of $f_1$ in the released CSV
exhibits a symmetric-exponent prime that exceeds every blocker.

\section*{Acknowledgements}
The author thanks the SageMath and PARI/GP development teams for
their freely available software.

\bibliographystyle{amsalpha}
\bibliography{paper2}

\end{document}